\documentclass[12pt,reqno]{amsart}

\usepackage[ansinew]{inputenc}
\usepackage{graphicx}
\usepackage[bookmarksnumbered,colorlinks, linkcolor=blue, citecolor=red, pagebackref, bookmarks, breaklinks]{hyperref}

\newtheorem{thm}{Theorem}[section]
\newtheorem{cor}[thm]{Corollary}
\newtheorem{lem}[thm]{Lemma}

\newtheorem{defi}[thm]{Definition}

\theoremstyle{remark}

\numberwithin{equation}{section}

\def\diver{\mathop{\text{\normalfont div}}}
\def\Dom{\mathop{\text{\normalfont Dom}}}

\newcommand{\R}{\mathbb{R}}
\newcommand{\N}{\mathbb{N}}

\newcommand{\G}{\mathcal{G}}

\newcommand{\M}{\mathcal{M}}
\newcommand{\ve}{\varepsilon}

\newcommand{\LL}{\mathcal{L}}

\newcommand{\average}{{\mathchoice {\kern1ex\vcenter{\hrule height.4pt
width 6pt depth0pt} \kern-9.7pt} {\kern1ex\vcenter{\hrule
height.4pt width 4.3pt depth0pt} \kern-7pt} {} {} }}

\def\R{\mathbb{R}}

\begin{document}

\title[$m-$Laplacian without the $\Delta_2$ condition]{Ljusternik-Schnirelmann eigenvalues for the fractional $m-$Laplacian without the $\Delta_2$ condition}

\author{Juli\'an Fern\'andez Bonder}

\address[JFB]{Instituto de C\'alculo -- CONICET and
Departamento de Matem\'atica, FCEN -- Universidad de Buenos Aires. 
Ciudad Universitaria, Pabell\'on I, C1428EGA, Av. Cantilo s/n
Buenos Aires, Argentina}

\email{jfbonder@dm.uba.ar}

\author{Juan F. Spedaletti}

\address[JS]{Departamento de Matem\'atica, FCFMyN, Universidad Nacional de San
Luis \hfill\break\indent Instituto DE Matem\'atica Aplicada San
Luis, IMASL, CONICET. \hfill\break\indent Italia avenue 1556,office
5, San Luis (5700), San Luis, Argentina.}

\email{jfspedaletti@unsl.edu.ar}



\begin{abstract}
In this work we analyze the eigenvalue problem associated to the fractional $m-$Laplacian, defined as
$$
(-\Delta_m)^s u(x):=2\text{p.v.}\int_{\R^n} m\left(\frac{|u(x)-u(y)|}{|x-y|^s}\right)\frac{(u(x)-u(y))}{|u(x)-u(y)|}\frac{dy}{|x-y|^{n+s}},
$$
This operator serves as a model for nonlocal, nonstandard growth diffusion problems. In contrast to previous analyses, we explore the eigenvalue problem without presuming the $\Delta_2$ condition on $M$ -- the primitive function of $m$. Our results show the existence of a sequence of eigenvalues $\lambda_k\to\infty$. This research contributes to advancing our understanding of nonlocal diffusion models, specifically those characterized by the fractional $m-$Laplacian, by relaxing the constraints imposed by the $\Delta_2$ condition.
\end{abstract}

\maketitle
\tableofcontents

\section{Introduction and main results}\label{Intro}

Eigenvalue problems stand as some of the most extensively investigated challenges within Partial Differential Equations. This interest arises both from their innate relevance to a wide array of natural phenomena, spanning vibrating membranes, quantum physics, and signal processing, among others, and from their intrinsic significance. Consider an open and bounded domain $\Omega\subset \R^n$. The classical Courant minimax principle ensures the existence of an infinite sequence of eigenvalues $\{\lambda_k\}_{k\in \N}$ for the classical Dirichlet eigenvalue problem:
$$
\begin{cases}
-\Delta u = \lambda u & \text{in }\Omega\\
u=0 & \text{in } \partial\Omega,
\end{cases}
$$
where $\lambda_k\to\infty$ (refer to \cite{TaLa} and \cite{Zei1}). The applications of this problem extend across various branches of mathematics and natural sciences.

In recent times, attention has shifted to nonlinear extensions and variations of the eigenvalue problem associated with the Laplacian. One of the most extensively studied is the eigenvalue problem for the $p-$Laplacian:
$$
\begin{cases}
-\Delta_p u = \lambda |u|^{p-2}u & \text{in }\Omega\\
u=0 & \text{in } \partial\Omega,
\end{cases}
$$
introduced by \cite{L} (also discussed in \cite{L2, L3}). For this problem,  the classical Ljusternik--Schnirelmann theory applied to the functionals $F$ and $G$ defined on $W_0^{1,p}(\Omega)$:
$$
F,\,G\colon W_0^{1,p}(\Omega)\to\R
$$ 
$$
F(u)=\int_\Omega |\nabla u|^p\,dx\qquad  \text{ and }\qquad G(u)=\int_\Omega|u|^p\,dx
$$
yields a sequence of eigenvalues $\{\lambda_k\}_{k\in\N}$ with $\lambda_k\to\infty$. Crucially, in this case, the space $W_0^{1,p}(\Omega)$ must be reflexive and separable, and the corresponding functionals $F$ and $G$ must be differentiable (see \cite{FuNe,FuNeSoSo,Zei1}).

Another interesting nonlinear eigenvalue problem arises with the $m-$Laplacian operator, defined as
$$
\Delta_m u = \diver\left(\frac{m(|\nabla u|)}{|\nabla u|}\nabla u\right),
$$
where $m\colon \R_+\to\R_+$ is a nondecreasing function. This operator generalizes the $p-$Laplacian operator when $m(t)=t^{p-1}$.

Consequently, the eigenvalue problem is given by
\begin{equation}\label{intro.1}
\begin{cases}
-\Delta_m u=\lambda g(u)& \text{in }\Omega\\
u=0 & \text{in }\partial\Omega.
\end{cases}
\end{equation}
Here, the function $g\colon \R\to \R$ satisfies certain growth conditions. What makes these operators particularly appealing for applications is the potential for distinct behaviors in diffusivity when $|\nabla u|\ll 1$ and $|\nabla u|\gg 1$, a phenomenon known in the literature as {\em nonstandard growth} elliptic operators. See \cite{Lieberman}.

A key factor in addressing such problems is the primitive function of $m$, denoted as $M(t)=\int_0^t m(s)\,ds$. When $M$ satisfies the so--called $\Delta_2-$condition, that is
$$
M(2t)\leq C M(t),
$$
for some constant $C>1$ and all $t\geq T_0$, then the eigenvalue problem \eqref{intro.1} inherits many properties from the $p-$Laplacian case. In such instances, the Ljusternik--Schnirelmann theory can be applied seamlessly to establish the existence of a sequence of eigenvalues for \eqref{intro.1}.

However, when $M$ does not satisfy the $\Delta_2-$condition, the situation becomes significantly more intrincate. In \cite{Ti}, the author analyzed problem \eqref{intro.1} and using ideas from \cite{Gossez} and employing a Galerkin--based approximation method, the author successfully overcome the absence of the $\Delta_2-$condition. Subsequently, the Ljusternik--Schinrelmann theory was applied, resulting in the identification of an infinite sequence of eigenvalues for \eqref{intro.1}.

\medskip

In recent years, nonlocal diffusion models have garnered considerable attention due to their diverse and novel applications in the natural sciences. These operators naturally arise in the context of stochastic L\'evy processes with jumps and have been extensively investigated from both probabilistic and analytical perspectives. Its applications range from physics, where it describes nonlocal interactions in materials, to finance, where it captures the memory effect in stochastic processes, and to image processing and ecology, where it accounts for spatial interactions over long distances, as documented in works such as \cite{Ap04, CT16, ST94} and the references therein. For the mathematical background from the partial differential equation (PDE) perspective adopted in this paper, readers can refer to \cite{BV, Garofalo}.

Arguably, one of the most significant nonlocal operators is the fractional Laplacian, defined as
$$
(-\Delta)^s u(x)=\text{p.v.}\int_{\R^n}\frac{u(x)-u(y)}{|x-y|^{n+2s}}\,dy, 
$$
where $s\in (0,1)$ is the fractional parameter. The associated eigenvalue problem takes the form
$$
\begin{cases}
(-\Delta)^s u = \lambda u & \text{in }\Omega\\
u=0 & \text{in } \R^n\setminus\Omega,
\end{cases}
$$
which can be analyzed using standard methods of functional analysis. This fractional Laplacian has proven to be a powerful tool in capturing nonlocal interactions and long-range dependencies, making it an invaluable tool for modeling phenomena characterized by anomalous diffusion.

In the realm of nonlocal diffusion models, numerous nonlinear generalizations of the fractional Laplacian eigenvalue problem have been explored in the literature. One particularly well-studied extension is encapsulated by the fractional $p-$Laplacian operator, defined as
$$
(-\Delta_p)^s u(x)=\text{p.v.}\int_{\R^n}\frac{|u(x)-u(y)|^{p-2}(u(x)-u(y))}{|x-y|^{n+2s}}\,dy.
$$
It is noteworthy that when $p=2$, the fractional $p-$Laplacian reduces to the standard fractional Laplacian.

The eigenvalue problem associated with the fractional $p-$Laplacian has been a subject of investigation by various authors in recent years. Notable contributions include works by \cite{BPS, FBSS, FP}, among others. This line of research delves into understanding the spectral properties and the behavior of solutions for this nonlinear nonlocal eigenvalue problem. The fractional $p-$Laplacian offers a versatile framework that extends the capabilities of the standard fractional Laplacian by incorporating additional nonlinearity through the power $p$. These developments hold promise for applications in modeling complex phenomena where both nonlocal interactions and nonlinear effects play crucial roles. The exploration of such nonlocal operators enriches the mathematical tools available for describing a wide range of phenomena in different scientific disciplines.

In \cite{FBS}, the authors introduced a fractional counterpart of the $m-$Laplacian, offering a model for nonlocal, nonstandard growth diffusion problems. Specifically, for an increasing and continuous function $m(t)$, the fractional $m-$Laplacian operator is defined as
$$
(-\Delta_m)^s u(x) = \text{p.v.} \int_{\R^n} m\left(\frac{|u(x)-u(y)|}{|x-y|^s}\right)\frac{(u(x)-u(y))}{|u(x)-u(y)|}\frac{dy}{|x-y|^{n+s}}.
$$
It is worth noting that when $m(t)=t^{p-1}$, this fractional $m-$Laplacian reduces to the fractional $p-$Laplacian. Subsequent to the pioneering work of \cite{FBS}, numerous studies exploring this operator have emerged, as evidenced by works such as \cite{ABC, FBSV2, MSV}, and references therein.

The associated eigenvalue problem for the fractional $m-$Laplacian is given by
\begin{equation}\label{intro.2}
\begin{cases}
(-\Delta_m)^s u = \lambda g(u) & \text{in }\Omega\\
u=0 & \text{in } \R^n\setminus\Omega.
\end{cases}
\end{equation}
For an alternative eigenvalue problem associated with this operator, readers are directed to \cite{FBSV1}.

Previous studies of Problem \eqref{intro.2}, such as those found in \cite{BOS, S}, assumed the $\Delta_2-$condition on $M(t)$. Notably, \cite{SV} stands as the sole work known to address \eqref{intro.2} without imposing the $\Delta_2-$condition, demonstrating the existence of a first eigenvalue for this problem.

Thus, the primary focus of this article lies in the investigation of Problem \eqref{intro.2} without relying on the $\Delta_2-$condition for the function $M(t)$. Our main result can be succinctly summarized as follows: 
\begin{thm}
Under suitable assumptions on $\Omega$, $m(t)$, and $g(t)$ {\em without necessitating the $\Delta_2-$condition on $M(t)$} there exists a sequence $\{\lambda_k\}_{k\in\mathbb{N}}$ of eigenvalues for \eqref{intro.2}. Moreover, $\lambda_k\to\infty$ as $k\to\infty$.
\end{thm}
For a precise and detailed statement of this result, please refer to Theorem \ref{teo.main} in Section \ref{section.4}. This contribution marks a significant advancement in our understanding of nonlocal eigenvalue problems, specifically those associated with the fractional $m-$Laplacian, by extending the analysis beyond the constraints of the $\Delta_2-$condition.

\section{Preliminaries}\label{prel}

In this section we present some preliminary definitions needed for the rest of the paper. The first subsection is well know and does not contain any new result being the book \cite{KR} the standard reference for the subject. The second subsection contains the definitions and basic results regarding fractional Orlicz-Sobolev spaces. See for instance \cite{BS} where these spaces were introduced and \cite{ACPS,ACPS2} where several properties of these spaces were analyzed whitout requiring the $\Delta_2-$condition. The third subsection recall the definition of complemantary pairs intruduced in \cite{Gossez} and construct a complementary pair in the context of fractional Orlicz-Sobolev spaces. Finally in the last subsection we recall an abstract result due to \cite{Ti} that will be helpful in the sequel.

\subsection{Young functions and Orlicz spaces}

Let $M\colon \R\to\R$ be a function, such that $M$ is even, convex and continuous,  $M(t)>0$ for $t>0$, $M(t)/t\to 0$ as $t\to 0$ and $M(t)/t\to\infty$ as $t\to\infty$. Such a function $M$ is called a {\em Young function} if it can be written as
$$
M(t) = \int_0^{|t|} m(s)\, ds,
$$
for $m\colon [0,\infty)\to [0,\infty)$ increasing, right continuous, $m(t)=0$ if and only if $t=0$ and $m(t)\to\infty$ as $t\to\infty$.

It will be helpful to extend the function $m$ to the entire real line by oddness, that is
$$
m(t)=\frac{m(|t|)}{|t|}t.
$$
We recall now some basic definitions on Orlicz spaces that can be found, for instance, in \cite{KR}.

Let $U\subset \R^N$ be a bounded domain and let $\mu$ be a Borel measure in $U$. The Orlicz class $\LL_M(U, d\mu)$ is defined as
$$
\LL_M(U,d\mu) := \left\{u\colon U\to \R,\ \text{measurable}\colon \int_U M(u)\, d\mu < \infty\right\}.
$$
The Orlicz space $L_M(U,d\mu)$ is then define as the linear hull of $\LL_M(U,d\mu)$. It follows that $L_M(U,d\mu)$ can be characterized as
$$
L_M(U,d\mu) = \left\{ u\colon U\to \R,\ \text{$\mu-$measurable}\colon \int_U M\left (\frac{u}{k}\right )\, d\mu < \infty, \text{ for some } k>0\right\}.
$$
This space is a Banach space when it is equipped, for instance, with the {\em Luxemburg norm}, i.e.
$$
\|u\|_{L_M(U,d\mu)} = \|u\|_{M,U,d\mu} = \|u\|_{M,d\mu} := \inf\left\{ k>0\colon \int_U M\left(\frac{u}{k}\right)\, d\mu \leq 1\right\}.
$$
A well known and interesting fact is that $\LL_M(U,d\mu) = L_M(U,d\mu)$ if and only if $M$ satisfies the so-called $\Delta_2-$condition, i.e.
\begin{equation}\label{Delta2}
M(2t)\le CM(t),\quad \text{for } t\ge T.
\end{equation}
Also, the Orlicz space $L_M(U,d\mu)$ is separable, if and only if $M$ satisfies \eqref{Delta2}.

Next, we define the space $E_M(U,d\mu)$ as the closure of bounded $\mu-$measurable functions in $L_M(U,d\mu)$, in the case $\mu(U)=\infty$ the space $E_M(U,d\mu)$ is the closure in $L_M(U,d\mu)$ of bounded $\mu-$measurable functions with bounded support. Again, $E_M(U,d\mu)=L_M(U,d\mu)$ if and only if $M$ satisfies \eqref{Delta2}.

So, in general, we have
$$
E_M(U,d\mu)\subset \LL_M(U,d\mu)\subset L_M(U,d\mu),
$$
with equalities if and only if $M$ satisfies \eqref{Delta2}.

Observe that $E_M(U,d\mu)$ and $L_M(U,d\mu)$ are Banach spaces and $\LL_M(U,d\mu)$ is a convex set.

\bigskip

Given a Young function $M$, we define its complementary function $\bar M$ as
$$
\bar M(t) := \sup\{ \tau|t|-M(\tau)\colon \tau\geq 0\}.
$$
Observe that $\bar M$ is also a Young function and is the optimal function in the Young inequality
\begin{equation}\label{young}
\tau t\le M(t) + \bar M(\tau),
\end{equation}
for all $\tau,t$. Observe that equality in \eqref{young} is achieved if and only if $\tau=m(t)$ sign $t$ or $t=\bar m(\tau)$ sign $\tau$ where $\bar{m}(t)$ is the derivative of $\bar{M}(t)$. 

It follows directly from \eqref{young} that if $u\in L_M(U,d\mu)$ and $v\in L_{\bar M}(U,d\mu)$, then $uv\in L^1(U)$ and
$$
\int_U |uv|\, d\mu \le 2 \|u\|_M \|v\|_{\bar M}.
$$
This fact allows one to define in $L_M(U,d\mu)$ the topology $\sigma(L_M,L_{\bar M})$ and it follows that $E_M(U,d\mu)$ is dense in $L_M(U,d\mu)$ in this topology.

It is easy to check that $\bar{\bar M} = M$. The Orlicz space $L_{\bar M}(U,d\mu)$ is the dual space of $E_M(U,d\mu)$ and so $L_M(U,d\mu)$ is reflexive if and only if $M$ and $\bar M$ satisfy \eqref{Delta2}.

Finally, given $M$ a Young function, we define
\begin{equation}\label{Domm}
\Dom(m) := \{u\in L_M(U,d\mu)\colon m(|u|)\in L_{\bar M}(U,d\mu)\}.
\end{equation}
It can be checked that $E_M(U,d\mu)\subset \Dom(m)\subset \LL_M(U,d\mu)$ and hence, $\Dom(m)=L_M(U,d\mu)$ if and only if $M$ satisfies \eqref{Delta2}. Moreover, the map $u\mapsto m(u)$ from $E_M(U,d\mu)$ to $L_{\bar M}(U,d\mu)$ is continuous if and only if $\bar M$ satisfies \eqref{Delta2}.

\subsection{Fractional Orlicz-Sobolev spaces}
In the product space $\R^n\times\R^n = \R^{2n}$ we define de measure
$$
d\nu_n := \frac{dxdy}{|x-y|^n}.
$$
Observe that this is a Borel measure and that if $K\subset \R^{2n}\setminus \Delta$ is compact, then $\nu_n(K)<\infty$, where $\Delta\subset \R^n\times\R^n$ is the diagonal $
\Delta := \{(x,x)\colon x\in\R^n\}$.

We will consider two Orlicz spaces $L_M(U,d\mu)$. One with $U=\R^n$ and $d\mu=dx$ (the Lebesgue measure) and other with $U=\R^{2n}$ and $d\mu = d\nu_n$. 

We will use the notations
$$
\LL_M = \LL_M(\R^n, dx), \qquad L_M = L_M(\R^n, dx),\qquad E_M = E_M(\R^n, dx);
$$
$$
\LL_M(\nu_n) = \LL_M(\R^{2n}, d\nu_n), \quad L_M(\nu_n) = L_M(\R^{2n}, d\nu_n),\quad E_M(\nu_n) = E_M(\R^{2n}, d\nu_n).
$$
Now, given a fractional parameter $s\in (0,1)$, we introduce the notation for the H\"older quotient of a function $u\colon\R^n\to \R$.
$$
D^su(x,y) := \frac{u(x)-u(y)}{|x-y|^s}.
$$
Then $D^su\colon \R^{2n}\setminus \Delta\to \R$.

Now, with all the notation introduced, the fractional Orlicz-Sobolev spaces are defined as
$$
W^sL_M := \{u\in L_M\colon D^s u\in L_M(\nu_n)\}
$$
and
$$
W^sE_M :=\{u\in E_M\colon D^su\in E_M(\nu_n)\}.
$$
These spaces are naturally equipped with the norms
$$
\|u\|_{s,M} = \|u\|_M + \|D^su\|_{M,\nu_n}.
$$
Also, these spaces can be isometrically identified as closed subspaces of $L_M\times L_M(\nu_n)$ and $E_M\times E_M(\nu_n)$ respectively using the map
$$
u\mapsto (u, D^su).
$$
Now, given $\Omega\subset\R^n$ a bounded open set, the space $W^s_0L_M(\Omega)$ is then defined as the closure of ${\mathcal D}(\Omega)$ in $W^sL_M$ with respect to the topology $\sigma(L_M\times L_M(\nu_n), E_{\bar M}\times E_{\bar M}(\nu_n))$. The space $W_0^sL_M(\Omega)$ is equipped with the norm $\|u\|_{W_0^sL_M(\Omega)}=\|u\|_{M,\Omega,dx}+\|D^su\|_{M,\nu_n}$ and by Poincar\'e's inequality (see \cite[Corollary 6.2]{FBS}) we can consider the space $W_0^sL_M(\Omega)$ with the equivalent norm $\|D^su\|_{M,\nu_n}$.

The space $W^s_0E_M$ is defined as the closure of ${\mathcal D}(\Omega)$ in $W^sE_M$ in norm topology.

In order to define the dual spaces, we need to introduce the notion of fractional divergence. See \cite{FBSaRi}.

Given $F\in L_{\bar{M}}(\nu_n)$, the fractional divergence of $F$ is defined as
\begin{align*}
{\mathrm{div}}^s F(x) &:= \text{p.v.} \int_{\R^n} \frac{F(y,x)-F(x,y)}{|x-y|^{n+s}}\, dy \\
&= \lim_{\ve\to 0} \int_{\R^n\setminus B_\ve(x)}\frac{F(y,x)-F(x,y)}{|x-y|^{n+s}}\, dy.
\end{align*}
In \cite{FBPLS} it is shown that for $F\in L_{\bar M}(\nu_n)$, then ${\mathrm{div}}^s F\in (W^s_0L_M(\Omega))^*$ and the following {\em fractional integration by parts formula} holds
$$
\langle {\mathrm{div}}^s F, u\rangle = \iint_{\R^{2n}} F D^s u\, d\nu_n.
$$
So, we define the following spaces of distributions
$$
W^{-s}L_{\bar M}(\Omega) := \{\phi\in {\mathcal D}'(\Omega)\colon \phi = f + \mathrm{div}^sF \text{ with } f\in L_{\bar M},\ F\in L_{\bar M}(\nu_n)\}
$$
$$
W^{-s}E_{\bar M}(\Omega) := \{\phi\in {\mathcal D}'(\Omega)\colon \phi = f + \mathrm{div}^sF \text{ with } f\in E_{\bar M},\ F\in E_{\bar M}(\nu_n)\}.
$$
Recall that since $E_{\bar{M}}$ and $E_{\bar{M}}(\nu_n)$ are separable then $W^{-s}E_{\bar M}(\Omega)$ is also separable.

These spaces are endowed with the usual quotient norms,
$$
\|\phi\|_{-s, \bar M} := \inf\{\|f\|_{\bar M} + \|F\|_{\bar M, \nu_v}\colon \phi=f+\mathrm{div}^s F\}.
$$

\subsection{Complementary systems}\label{complementary}

In \cite{Do, DoTru} the authors introduce the notion of complementary systems in order to work in spaces without the usual reflexivity assumption.

Let $Y$ and $Z$ be real Banach spaces with a duality pairing $\langle\cdot,\cdot\rangle$. Let $Y_0\subset Y$ and $Z_0\subset Z$ be closed and separable subspaces. We say $(Y, Y_0; Z, Z_0)$ is a complementary system if $Y_0^*=Z$ and $Z_0^*=Y$ (where equality is understood in the sense of a natural isometry via the duality pairing).

The first natural example of a complementary system is $Y=L_M$, $Y_0=E_M$, $Z=L_{\bar M}$ and $Z_0=E_{\bar M}$.

We use the notation $(Y, Y_0; Z, Z_0)$ for a complementary system. Observe that it is immediate to see that
\begin{equation}\label{LMEM}
\left (L_M\times L_M(\nu_n),E_M\times E_M(\nu_n); L_{\bar M}\times L_{\bar M}(\nu_n),E_{\bar M}\times E_{\bar M}(\nu_n)
\right)
\end{equation}
is also a complementary system.

In \cite{Gossez} the author provides with a general method to generate complementary systems from a previous one. More precisely
\begin{lem}[\cite{Gossez}, Lemma 1.2]\label{lemaGossez}
Given a complementary system $(Y, Y_0; Z, Z_0)$ and a closed subspace $E$ of $Y$, define $E_0= E \cap Y_0$, $F = Z/E_0^\perp$ and $F_0 = Z_0/E_0^\perp$.

Then, the pairing $\langle\cdot,\cdot\rangle$ between $Y$ and $Z$ induces a pairing between $E$ and $F$ if and only if $E_0$ is $\sigma(Y,Z)$ dense in $E$. In this case, $(E, E_0; F, F_0)$ is a complementary system if $E$ is $\sigma(Y, Z_0)$ closed, and conversely, when $Z_0$ is complete, $E$ is $\sigma(Y, Z_0)$ closed if $(E, E_0; F, F_0)$ is a complementary system.
\end{lem}

Using this Lemma, in \cite{Gossez} it is shown that 
$$
(W^1_0L_M(\Omega),W^1_0E_M(\Omega); W^{-1}L_{\bar M}(\Omega), W^{-1}E_{\bar M}(\Omega))
$$
is a complementary system when the domain $\Omega$ satisfies the {\em segment property} (See Definition \ref{segment} for a precise statement). 

Let us now see that
\begin{equation}\label{CS}
(W^s_0L_M(\Omega),W^s_0E_M(\Omega); W^{-s}L_{\bar M}(\Omega), W^{-s}E_{\bar M}(\Omega))
\end{equation}
is also a complementary system under the same assumptions on $\Omega$.

In fact, since \eqref{LMEM} is a complementary system, we use Lemma \ref{lemaGossez} to generate \eqref{CS} from \eqref{LMEM}.

So we take $E=W^s_0L_M(\Omega)$ and $E_0=E\cap Y_0=W^s_0E_M(\Omega)$. It is also easy to see that
$F=Z/E_0^\perp = W^{-s}L_{\bar M}(\Omega)$ and $F_0 =W^{-s}E_{\bar M}(\Omega)$. So in order to see that \eqref{CS} is a complementary system it remains to check that $W^s_0E_M(\Omega)$ is $\sigma(L_M\times L_M(\nu_n),L_{\bar{M}}\times L_{\bar{M}}(\nu_n))$ dense in $W^s_0L_M(\Omega)$ and that $W^s_0L_M(\Omega)$ is $\sigma(W^s_0L_M(\Omega), W^{-s}E_{\bar M}(\Omega))$ closed.

Now, $W^s_0L_M(\Omega)$ is $\sigma(W^s_0L_M(\Omega), W^{-s}E_{\bar M}(\Omega))$ closed by definition. The proof of the density of $W^s_0E_M(\Omega)$ in $W^s_0L_M(\Omega)$ with respect to the $\sigma(L_M\times L_M(\nu_n),L_{\bar{M}}\times L_{\bar{M}}(\nu_n))$ topology follows similarly as in \cite[Theorem 1.3]{Gossez}. All these details are collected in Appendix \ref{ap.density} for the reader convenience. See Theorem \ref{teo.density}.

\subsection{An abstract result}
In this subsection, we recall an abstract result from \cite{Ti} where the author construct in a complementary system a sequence of projector operators converging to the identity. This abstract result will be of critical importance in the application of the Ljusternik-Schnirelmann method.

\begin{thm}[\cite{Ti}, Theorem 3.1]\label{teo.abstracto}
Assume $(E,E_0;F,F_0)$ is a complementary system, the norm $\|\cdot\|_F$ is dual to $\|\cdot\|_{E_0}$, the norm $\|\cdot\|_E$ is dual to $\|\cdot\|_{F_0}$ and $V\subset E_0$ is a norm-dense linear subspace. Then there exists a sequence of mappings $P_k\colon E_0 \to E_0, k=1,2,\dots$ satisfying
\begin{itemize}
\item $P_k$ is odd and norm-continuous for all $k=1,2,\dots$
\item $P_k(E_0)$ is contained in a finite-dimensional subespace of $V$ for all $k=1,2,\dots$
\item If $\{u_k\}\in E_0$ and $u_k\to u\in E$ for $\sigma(E,F_0)$, then $P_k(u_k)\to u$ for $\sigma(E,F_0)$.
\item If $\{u_k\}\in E_0$ and $u_k\to u\in E$ strongly, then $\|P_k(u_k)\|_E\to \|u\|_E$.  
\end{itemize}	
\end{thm}

\section{The fractional $m-$laplacian $(-\Delta_m)^s$}
In this section we introduce the integro--differential operator appearing in our eigenvalue problem \eqref{intro.2}. This operator was first introduced in \cite{BS} and was analyze in the case where the Young function $M$ satisfies the $\Delta_2-$condition.

Let $M$ be a Young function and $\Omega\subset \R^n$ be a bounded, open set with the segment property. Recall that for $0<s<1$ the fractional $m-$Laplacian of a function $u$ is defined as
\begin{align*}
(-\Delta_m)^s u(x)&:=2\,\text{p.v}.\int_{\R^{n}}m(D^su)\,\frac{dy}{|x-y|^{n+s}}
\\
&=2\lim_{\varepsilon\downarrow 0} \int_{|x-y|\geq \varepsilon} m(D^su)\,\frac{dy}{|x-y|^{n+s}}.
\end{align*}
Let us see now that this operator is well defined between the spaces $\Dom((-\Delta_m)^s)$ and $W^{-s}L_{\bar{M}}(\Omega)$ respectively where 
$$
\Dom((-\Delta_m)^s) := \{u\in W^s_0L_M(\Omega)\colon m(D^su)\in L_{\bar M}(\nu_n)\}. 
$$
To do this we consider for $\varepsilon >0$
$$
(-\Delta_m)_\varepsilon^s u(x):=2\int_{|x-y|\geq \varepsilon}m(D^su)\,\frac{dy}{|x-y|^{n+s}}.
$$
\begin{thm}
Let $0<s<1$ be fixed. For $u\in \Dom((-\Delta_m)^s)$ the limit $(-\Delta_m)^s u:=\lim_{\varepsilon\downarrow 0}(-\Delta_m)_\varepsilon^s u$ exists in $W^{-s}L_{\bar{M}}(\Omega)$, that is 
$$
\langle (-\Delta_m)^s u,v \rangle:=\lim_{\varepsilon \downarrow 0}\langle (-\Delta_m)_\varepsilon^s u,v\rangle<\infty,
$$
forall $v\in W_0^sE_M(\Omega)$.

Moreover the following representation formula holds
$$
\langle (-\Delta_m)^su, v\rangle=\iint_{\R^{2n}}m(D^su) D^sv\,d \nu_n,
$$ 	
for all $v\in W_0^sE_{M}(\Omega)$.
\end{thm}
\begin{proof}
Let $0<\varepsilon<1$. We begin by proving that $(-\Delta_m)_\varepsilon^s u\in L_{\bar{M}}$ for $u\in \Dom((-\Delta_m)^s)$. If $u\in \Dom((-\Delta_m)^s)$ then there exists a constant $k>0$ such that 
$$
\iint_{\R^{2n}}\bar{M}\left (\frac{m(D^su)}{k}\right)d\nu_n<\infty,
$$
therefore by Jensen's inequality
\begin{align*}
\int_{\R^n}\bar{M}\left (\frac{(-\Delta_m)_\varepsilon^s u}{\left ({\frac{2k\varepsilon^{-s}\omega_{n-1}}{s}}\right )} \right )\,dx&=\int_{\R^n}\bar{M}\left (\frac{\int_{|x-y|\geq \varepsilon}\frac{m\left (\frac{u(x)-u(y)}{|x-y|^s} \right )}{k}\,\frac{dy}{|x-y|^{n+s}}}{\frac{\varepsilon^{-s}\omega_{n-1}}{s}} \right ) \,dx
\\
&\leq\frac{\varepsilon^{s}s}{\omega_{n-1}}\int_{\R^n}\int_{|x-y|\geq\varepsilon} \bar{M}\left (\frac{m\left (\frac{u(x)-u(y)}{|x-y|^s} \right )}{k}\right )\,\frac{dy}{|x-y|^{n+s}}\,dx
\\
&\leq \frac{\varepsilon^s s}{\omega_{n-1}}\iint_{\R^{2n}}\bar{M}\left( \frac{m(D^s u)}{k}\right)\,d\nu_n
\\
&<\infty,
\end{align*}
where $\omega_{n-1}$ denotes the measure of the $(n-1)-$dimensional sphere $S^{n-1}$. The above inequality implies $(-\Delta_m)_\varepsilon^s u\in L_{\bar{M}}$ .

Let $v\in W_0^sE_M(\Omega)$, using Fubini's theorem and change variables we obtain
\begin{align*}
\langle (-\Delta_m)_\varepsilon^s u,v\rangle &=2\int_{\R^n}\int_{|x-y|\geq \varepsilon}m\left ( \frac{u(x)-u(y)}{|x-y|^s}\right ) v(x)\,\frac{dy}{|x-y|^{n+s}} \,dx	
\\
&=2\int_{\R^n}\int_{|x-y|\geq \varepsilon}m\left ( \frac{u(y)-u(x)}{|x-y|^s}\right ) v(y)\,\frac{dy}{|x-y|^{n+s}} \,dx,	
\end{align*}
and so 	
\begin{align*}
\langle (-\Delta_m)_\varepsilon^s u,v\rangle &=\int_{\R^n}\int_{\R^n}m\left ( \frac{u(x)-u(y)}{|x-y|^s}\right ) \left (\frac{v(x)-v(y)}{|x-y|^s}\right )\,\chi_{\{|x-y|\geq\varepsilon\}}(x,y)\,\frac{dxdy}{|x-y|^n}
\\
&=\iint_{\R^{2n}}m(D^s u)D^sv \chi_{\{|x-y|\geq \varepsilon\}}\,d\nu_n,
\end{align*}
since $u\in \Dom((-\Delta_m)^s)$ and $v\in W_0^s E_M(\Omega)$ we have $m(D^s u)\in L_{\bar{M}}(\nu_n)$ and $D^sv\in L_M(\nu_n)$ respectively, therefore by the dominated convergence theorem we conclude the proof.
\end{proof}
By our remarks after the definition of $\Dom(m)$, \eqref{Domm}, it follows that 
$$
W^s_0E_M(\Omega)\subset \Dom((-\Delta_m)^s)\subset W^s_0L_M(\Omega).
$$
Recall now that the monotonicity of $m$ implies that, for any $a,b\in\R$,
$$
(m(a)-m(b))(a-b)\ge (m(|a|)-m(|b|))(|a|-|b|)\ge 0,
$$
from where it follows that the operator $(-\Delta_m)^s$ is monotone. That is
$$
\langle (-\Delta_m)^s u - (-\Delta_m)^s v, u-v\rangle \ge 0,\quad \text{for } u,v\in \Dom((-\Delta_m)^s).
$$
Another key property of the fractional $m-$laplacian is that it is {\em pseudomonotone}, this is the content of the following theorem.
\begin{thm}\label{pseudo}
Let $\Omega\subset\R^n$ be a bounded domain that satisfies the segment property. If $\{u_i\}_{i\in\N}\subset \Dom((-\Delta_m)^s)$ is a sequence such that fulfills the conditions
$$
\begin{cases}	
&u_i\to u\text{ for }\sigma(W_0^sL_M(\Omega),W^{-s}E_{\bar{M}}(\Omega))\\
&(-\Delta_m)^su_i\to f\in W^{-s}L_{\bar{M}}(\Omega)\text{ for }\sigma(W^{-s}L_{\bar{M}}(\Omega),W_0^sE_M(\Omega))\\
&\limsup_{i\to \infty}\langle (-\Delta_m)^s u_i,u_i\rangle\leq \langle f,u\rangle
\end{cases}
$$
then 
$$
\begin{cases}
&u\in \Dom((-\Delta_m)^s)
\\
&(-\Delta_m)^s u=f
\\
&\langle (-\Delta_m)^s u_i,u_i\rangle\to  \langle f,u\rangle\text{ if }i\to \infty.
\end{cases}
$$
\end{thm}
It will be convenient to introduce the following notation. Given a function $u$, we denote the sets $\{R_j(u)\}_{j\in\N}$ as
$$
R_j(u):=\{(x,y)\in\R^{2n}\colon |(x,y)|\leq j\text{ and }|D^su(x,y)|\leq j\}.
$$
To prove Theorem \ref{pseudo} we need first the following two lemmas.
\begin{lem}\label{paraMinty} 
Let $u$ be a function in $W_0^sL_M(\Omega)$ and $v\in W_0^sE_M(\Omega)$ respectively. If $0<|\lambda|<1$ then for each $j\in\N$
$$
m(D^su+\lambda D^sv) D^sv,\quad m(D^s u) D^sv\in L^1(R_j(u),\nu_n).
$$ 
Moreover 
$$
 \lim_{\lambda\to 0}\iint_{R_j(u)}m(D^su+\lambda D^sv) D^sv\,d\nu_n=\iint_{R_j(u)} m(D^su) D^sv\,d\nu_n.
 $$	
\end{lem}
\begin{proof}
Let $v\in W_0^sE_M(\Omega)$ and $u\in W_0^sL_M(\Omega)$, then $m(2|D^sv|)\in L_{\bar{M}}(\nu_n)$ and $D^sv\in L_M(\nu_n)$ then there exists constants $k,\tilde{k}>0$ such that
\begin{equation}\label{c1}
\iint_{\R^{2n}}\bar{M}\left(\frac{m(2|D^sv|)}{k} \right)\,d\nu_n<\infty 
\end{equation}
and
\begin{equation}\label{c2}
\iint_{\R^{2n}}M\left ( \frac{|D^sv|}{\tilde{k}}\right )\,d\nu_n<\infty.
\end{equation}
We observe by using Young's inequality that
\begin{equation}\label{c3}
\begin{split}
|m(D^su+\lambda D^sv)D^sv|&\leq k\tilde{k}\left \{\bar{M}\left (\frac{m(|D^su|+|D^sv|)}{k}\right )+M\left(\frac{|D^sv|}{\tilde{k}}\right) \right\}
\\
&\leq k\tilde{k}\left \{\bar{M}\left (\frac{m(j+|D^sv|)}{k}\right )+M\left(\frac{|D^sv|}{\tilde{k}}\right) \right\}
\end{split}
\end{equation}
in $R_j(u)$.
Therefore for each $j$ 
\begin{align*}
&\iint_{R_j(u)}\left |\bar{M}\left ( \frac{m(j+|D^sv|)}{k}\right ) \right| \,d\nu_n=
\\
&\left\{ \iint_{\{|D^sv|\leq j\}\cap R_j(u)}+\iint_{\{|D^sv|>j\}\cap R_j(u)}\right\}\left |\bar{M}\left ( \frac{m(j+|D^sv|)}{k}\right ) \right| \,d\nu_n\leq 
\\
&\bar{M}(m(2j/k))|R_j(u)|+\iint_{\{|D^sv|>j\}\cap R_j(u)}\bar{M}\left(\frac{m(2|D^sv|)}{k} \right )\,d\nu_n\leq
\\
&\bar{M}(m(2j/k))|R_j(u)|+\iint_{\R^{2n}}\bar{M}\left ( \frac{m(2|D^sv|)}{k}\right )\,d\nu_n<\infty,
\end{align*}
from this inequality together with \eqref{c1}-\eqref{c3} it follows that $m(D^su+\lambda D^sv)D^sv \in L^1(R_j(u),d\nu_n)$. 

Moreover observe that in $R_j(u)$,
$$
|m(D^su)D^sv|\leq m(j)|D^sv|\leq \bar{M}(\tilde{k}m(j))+M\left ( \frac{|D^sv|}{\tilde{k}}\right ).
$$
So using \eqref{c2} it follows that $m(D^su)D^sv\in L^1(R_j(u),\nu_n)$. 

Finally, by using 
$$
|m(D^su+\lambda D^sv)D^sv|\leq k\tilde{k}\left \{\bar{M}\left (\frac{m(j+|D^sv|)}{k}\right )+M\left(\frac{|D^sv|}{\tilde{k}}\right) \right\}\in L^1(R_j(u),\nu_n),
$$
the fact
$$
m(D^su+\lambda D^sv)  D^sv\to m(D^s u) D^sv
$$
if $\lambda\to 0$ a.e. in $R_j(u)$
and the dominated convergence theorem we conclude the proof.
\end{proof}
\begin{lem}\label{parapseudo}
If there exists $u\in W_0^sL_M(\Omega)$ and $\phi \in L_{\bar{M}}(\nu_n)$ such that
\begin{equation}\label{cpseudo}
\iint_{R^{2n}}\left (m(W)-\phi\right)(W-D^s u)\,d\nu_n\geq 0,
\end{equation}
for all $W\in L^{\infty}(\R^{2n},d\nu_n)$ with compact support then $m(D^su)=\phi$ in $(W_0^sL_{M}(\Omega))^*$ that is
$$
\iint_{\R^{2n}}m(D^su)D^sv\,d \nu_n=\iint_{\R^{2n}}\phi D^sv\,d \nu_n\qquad \forall v\in W_0^sL_M(\Omega).
$$ 
\end{lem}
\begin{proof}
Let $w\in W_0^sL_M(\Omega)$ be such that $D^sw\in L^{\infty}(R_j(u),\nu_n)$. For $l\geq j$ we take 

$$
W\equiv D^sw \chi_{R_j(u)}-D^s u\chi_{R_j(u)}+D^su \chi_{R_l(u)}=D^{s,j}w-D^{s,j}u+D^{s,l}u,
$$ 
now using $W$ in \eqref{cpseudo} we obtain
\begin{equation}\label{des3.5}
\iint_{\R^{2n}}(m(D^{s,j}w-D^{s,j}u+D^{s,l}u)-\phi)((D^{s,j}w-D^{s,j}u+D^{s,l}u)- D^su	)d\nu_n\geq 0
\end{equation}
The left hand side in the above inequality can be written as
\begin{align*}
&\iint_{\R^{2n}}(m(D^{s,j}w-D^{s,j}u+D^{s,l}u)-\phi)(D^{s,j}w-D^{s,j}u)\,d\nu_n+
\\
&\iint_{\R^{2n}}m(D^{s,j}w-D^{s,j}u+D^{s,l}u)(D^{s,l}u- D^su)\,d\nu_n-
\\
&\iint_{\R^{2n}}\phi (D^{s,l}u- D^su)\,d\nu_n=
\\
& I+II+III.
\end{align*}
The first integral $I$ is zero outside $R_j(u)$, therefore
\begin{align*}
I=&\iint_{R_j(u)}(m(D^{s,j}w-D^{s,j}u+D^{s,l}u)-\phi)(D^{s,j}w-D^{s,j}u)\,d\nu_n
\\
=&\iint_{R_j(u)}(m(D^sw)-\phi)(D^sw-D^su)\,d\nu_n.
\end{align*}
For the second integral $II$ we observe that $D^{s,l}u-D^su$ is zero inside $R_l(u)$, and outside $R_l(u)$ we have $m(D^{s,j}w-D^{s,j}u+D^{s,l}u)=m(0)=0$ and so $II=0$. The third integral $III$ goes to zero as $l\to+\infty$. Hence, letting $l\to\infty$ in \eqref{des3.5}, we obtain
\begin{equation}\label{des3.6}
\iint_{R_j(u)}(m(D^s w)-\phi)(D^sw-D^su)\,d\nu_n\geq 0,
\end{equation}
$\forall w\in W_0^sL_M(\Omega)$ with $D^sw\in L^\infty(R_j(u),\nu_n)$.

Now let $v\in \mathcal{D}(\Omega)$, using \eqref{des3.6} and Lemma \ref{paraMinty} with $\lambda>0$ first with $w=u+\lambda v$ and then $w=u-\lambda v$ we have
$$
\iint_{R_j(u)}\left (m(D^su)-\phi \right ) D^sv\,d\nu_n=0\qquad \forall v\in\mathcal{D}(\Omega),
$$
taking limit $j\to\infty$ we have 
$$
\iint_{\R^{2n}}\left (m(D^su)-\phi \right ) D^sv\,d\nu_n=0\qquad \forall v\in\mathcal{D}(\Omega),
$$
and by density $\forall v\in W_0^sE_M(\Omega)$. Using the density of $W^s_0E_M(\Omega)$ in $W^s_0L_M(\Omega)$ with respect to the $\sigma(L_M\times L_M(\nu_n),L_{\bar{M}}\times L_{\bar{M}}(\nu_n))$ topology we conclude the proof.
\end{proof}
With this preliminaries we are ready to prove the pseudomonotonicity.
\begin{proof}[Proof of Theorem \ref{pseudo}]
Given a sequence $\{u_i\}_{i\in\N}\subset \Dom((-\Delta_m)^s)$ such that $u_i\to u\in W_0^sL_M(\Omega)$ for $\sigma(W_0^sL_M(\Omega),W^{-s}E_{\bar{M}}(\Omega))$, $(-\Delta_m)^su_i\to f\in W^{-s}L_{\bar{M}}(\Omega)$ for $\sigma(W^{-s}L_{\bar{M}}(\Omega),W_0^sE_M(\Omega))$ and 
\begin{equation}\label{limsup}
\limsup_{i\to \infty}\langle (-\Delta_m)^s u_i,u_i\rangle\leq \langle f,u\rangle.
\end{equation}
We must prove that $u\in \Dom((-\Delta_m)^s),(-\Delta_m)^s u=f$ and $\langle (-\Delta_m)^s u_i,u_i\rangle\to  \langle f,u\rangle$ if $i\to \infty$.

We prove first that the sequence $\{m(D^su_i)\}_{i\in\N}$ remains bounded in $L_{\bar{M}}(\nu_n)$. 

Using Young's inequality we have
$$
m(D^su_i)D^su_i=\bar{M}(m(D^su_i))+M(D^su_i), 
$$
then
\begin{align*}
\iint_{\R^{2n}}\bar{M}(m(D^su_i))\,d\nu_n &\leq \iint_{\R^{2n}}m(D^su_i)D^su_i\,d\nu_n 
\\
&=\langle (-\Delta_m)^s u_i,u_i\rangle
\\
&\leq|\langle (-\Delta_m)^s u_i,u_i\rangle|\leq C,
\end{align*}
by \eqref{limsup}. Therefore the sequence $\{m(D^su_i)\}_{i\in\N}$ is bounded in $L_{\bar{M}}(\nu_n)$ hence there exists a subsequence that we still denote $\{m(D^s u_i)\}_{i\in\N}$ and $\phi\in L_{\bar{M}}(\nu_n)$ such that $m(D^su_i)\to \phi$ in $\sigma(L_{\bar{M}}(\nu_n),E_M(\nu_n))$ combining this fact with the convergence $(-\Delta_m)^su_i\to f\in W^{-s}L_{\bar{M}}(\Omega)$ for $\sigma(W^{-s}L_{\bar{M}}(\Omega),W_0^sE_M(\Omega))$ implies that for any $v\in W_0^sE_M(\Omega)$,
\begin{align*}
\langle f,v \rangle	&=\lim_{i\to\infty}\iint_{\R^{2n}}m(D^su_i) D^sv\,d\nu_n
\\
&=\iint_{\R^{2n}}\phi D^sv\,d\nu_n.
\end{align*}
This formula togheter with the density of $W^s_0E_M(\Omega)$ in $W^s_0L_M(\Omega)$ with respect to the $\sigma(L_M\times L_M(\nu_n),L_{\bar{M}}\times L_{\bar{M}}(\nu_n))$ topology allow us to extend $f$ to the space $W_0^sL_M(\Omega)$.

Let $W\in L^{\infty}(\R^{2n},d\nu_n)$ with compact support, using the monotonicity property of $(-\Delta_m)^s$ 
\begin{equation}\label{mono}
\iint_{\R^{2n}}(m(D^su_i)-m(W))(D^su_i-W)\,d\nu_n\geq 0.	
\end{equation}
All the above discussion allow us to pass to the limit in \eqref{mono} and we obtain
$$
\iint_{\R^{2n}}(\phi-m(W))(D^su-W)\,d\nu_n\geq 0.
$$
It then follows from Lemma \ref{parapseudo} that $m(D^su)=\phi$ in $(W_0^sL_{M}(\Omega))^*$ and so by Young's inequality
\begin{align*}
\iint_{\R^{2n}}\phi D^su\,d\nu_n
&=\iint_{\R^{2n}}m(D^su)D^su\,d\nu_n
\\
&=\iint_{\R^{2n}}\bar{M}(m(D^su))\,d\nu_n+\iint_{\R^{2n}}M(D^su)\,d\nu_n
\\
&\geq \iint_{\R^{2n}}\bar{M}(m(D^su))\,d\nu_n
\end{align*}
the above says that $m(D^su)\in L_{\bar{M}}(\nu_n)$ and $u\in \Dom((-\Delta_m)^s)$. Also
$$
\langle f,v \rangle=\iint_{\R^{2n}}\phi D^sv\,d\nu_n=\iint_{\R^{2n}}m(D^su)D^sv\,d\nu_n=\langle (-\Delta_m )^s u,v \rangle,  
$$
for all $v\in W_0^sL_M(\Omega)$, then $(-\Delta_m)^su=f$ in $(W_0^sL_M(\Omega))^*$.

Finally, we prove that $\langle (-\Delta_m)^s u_i,u_i\rangle\to  \langle f,u\rangle$ for $i\to \infty$. Let 
$$
L=\liminf_{i\to \infty}\iint_{\R^{2n}}m(D^su_i)D^su_i\,d\nu_n,
$$
we only need to prove that $L\geq \langle f,u \rangle$. Using the monotonicity property again we have
$$
\iint_{R^{2n}}(m(D^su_i)-m(D^{s,j}u))(D^su_i-D^{s,j}u)\,d\nu_n\geq 0,
$$
where $D^{s,j}u=D^su\chi_{R_j(u)}$ taking limit in $i$ and rewriting we obtain
$$
L\geq \iint_{\R^{2n}}m(D^{s,j}u)(D^su-D^{s,j}u)\,d\nu_n+\iint_{\R^{2n}}m(D^su)D^{s,j}u\,d\nu_n.
$$
In $R_j(u)$ the factor $D^su-D^{s,j}u=0$, and in $(R_j(u))^c$ the factor $m(D^{s,j}u)=m(0)=0$ so the first integral in the above inequality is zero, then 
$$
L\geq \iint_{R_j(u)}m(D^su)D^su\,d\nu_n
$$
for arbitrary $j$ therefore
$$
L\geq \iint_{\R^{2n}}m(D^su)D^su\,d\nu_n=\langle f,u\rangle.
$$
This concludes the proof of the theorem.
\end{proof}

\section{The eigenvalue problem}\label{section.4}
In this section we study the main result of the paper, namely the existence of a sequence $\{(\lambda_k, u_k)\}_{k\in\N}$ of eigenpairs of the equation \eqref{intro.2} and, moreover, $\lambda_k\to\infty$ as $k\to\infty$. 

We say that $(\lambda,u)$ is an eigenpair of \eqref{intro.2} if 
\begin{equation}\label{soldebil}
\iint_{\R^{2n}}	m(D^su)D^s \phi\,d\nu_n=\lambda \int_{\Omega} g(u)\phi\,dx,
\end{equation}
for every $\phi\in \mathcal{D}(\Omega)$, provided that both integrals are defined.

The strategy of the proof is to apply the Ljusternik-Schnirelmann method that have been proved to be succesful in previous works. However, the lack of reflexivity of the spaces involved prevent us to apply directly the Ljusternik-Schnirelmann method. This fact was already observe by \cite{Ti} where the author is able to reduce the problem to a finite dimensional one and then pass to the limit. Here we apply the same idea to the context of fractional order spaces.

In fact after the work performed in the previous sections the ideas of \cite{Ti} can be applied almost straightforward. Howewer we include the details for the reader convenience and to make the paper selfcontained.  

Therefore we look for the existence of a sequence $(\lambda_k, u_k)\subset \R\times W^s_0E_M(\Omega)$ of eigenpairs for the problem \eqref{intro.2}. Observe that, since $u_k\in W^s_0E_M(\Omega)\subset \Dom((-\Delta_m)^s)$, $m(D^su)\in L_{\bar{M}}(\nu_n)$. Hence the left hand side of \eqref{soldebil} is well defined forall $\phi\in W_0^sE_M(\Omega)$.

On the function $g\colon \R\to\R$ we assume that is an odd and continuous function satisfying $g(t)t>0$ for all $t\neq 0$ and
\begin{equation}\label{g}
|g(t)|\le a_1 + a_2 m(a_3t)\qquad \text{ for all }t\geq 0,
\end{equation}
where $a_1,a_2$ and $a_3$ are positive constants. Observe that if $u\in W_0^s E_M(\Omega)$ the right hand side of \eqref{soldebil} is well defined.

Therefore, the main result of this paper reads as follows:
\begin{thm}\label{teo.main}
Let $\Omega\subset \R^n$ be an open and bounded domain with the segment property. Then there exists a sequence of eigenpairs $\{(\lambda_k, u_k)\}_{k\in\N}\subset \R_+\times W_0^s L_M(\Omega)$ of \eqref{intro.2}. Moreover, $\lambda_k \to +\infty$ and  $u_k\to 0$ in the topology $\sigma(W_0^sL_M(\Omega), W^{-s}E_{\bar{M}}(\Omega))$ as $k\to\infty$.
\end{thm}

We consider the even functionals $\M_s\colon D_{\M_s}\to\R$ and $\G\colon D_\G\to\R$ defined as
\begin{align}\label{M}
\M_s(u)&:=\iint_{\R^{2n}} M(D^su)\, d\nu_n\\
\G(u)&:= \int_\Omega G(u)\, dx,
\end{align}
where $G(t)=\int_0^{|t|}g(\tau)\,d\tau$ and
$$
D_{\M_s}:=\{u\in W^s_0L_M(\Omega)\colon \M(u)<\infty\} \quad \text{and}\quad D_\G:=\{u\in W^s_0L_M(\Omega)\colon \G(u)<\infty\}.
$$
It is clear that $W_0^{s}E_M(\Omega)\subset \Dom((-\Delta_m)^s)\subset D_{\M_s}\subset W_0^sL_M(\Omega)$ and both functionals $\M_s$ and $\G$ vanish only at zero.  

Let $B=\{\phi_1,\phi_2,\dots\}\subset {\mathcal D}(\Omega)$ be a countable linearly independent subset in $W_0^sE_M(\Omega)$ such that the linear hull is norm dense, and we define the sets $V$ and $V_k$ as the linear hull of the sets $B$ and $B_k:=\{\phi_1,\phi_2,\dots,\phi_k\}$ respectively. We denote the continuous pairing between $W^{-s}L_{\bar{M}}(\Omega)$ and $W^s_0E_M(\Omega)$ by $\langle \cdot, \cdot \rangle$, and the one between $V_k$ and $(V_k)^*$ by $\langle \cdot, \cdot \rangle_k$.

A straightforward calculation gives $\M_s, \G \in  C^1(V_k)$ and
\begin{align*}
\langle \M_s'(u),v\rangle_k &=\iint_{\R^{2n}}m(D^su) D^sv\,d \nu_n\qquad &\text{for all } u,v\in V_k
\\
\langle\G'(u),v\rangle_k &=\int_\Omega g(u) v\,dx\qquad &\text{for all } u,v\in V_k,
\end{align*}
for each $k=1,2,\dots$ 

By the growth condition \eqref{g} and the compact immersion $W^s_0L_M(\Omega) \subset\subset E_M(\Omega)$ (see Theorem \ref{teo.compact}),  it follows that $D_\G=W_0^sL_M(\Omega)$ and 
\begin{align*}
\G(u_k) &\to \G(u)\\
\langle \G'(u_k),u_k\rangle	&\to \int_\Omega g(u) u\, dx\\
\langle \G'(u_k),v\rangle &\to \int_\Omega g(u) v\, dx,
\end{align*}
whenever $u_k\in V_k, u_k\to u\in W_0^sL_M(\Omega)$ in $\sigma(W_0^sL_M(\Omega),W^{-s}E_{\bar{M}}(\Omega)),$ and $v\in V$.

A first step to attack the problem \eqref{intro.2} is to consider the following problem on the finite dimensional space $V_k$
\begin{equation}\label{eigenfinito}
\begin{split}
&\mathcal{M}_s^{'}(u)=\lambda \mathcal{G}'(u)\text{ in }V_k^*
\\
&u\in \mathcal{N}_k^s,\lambda\in \R, 
\end{split}
\end{equation}
where $\mathcal{N}_k^s=\{u\in V_k\colon \mathcal{M}_s(u)=1\}$.

Now, in order to study the problem \eqref{eigenfinito},
we use the following notation
\begin{align*}
\mathcal{N}^s	&=\{u\in W_0^sE_M(\Omega)\colon \M_s(u)=1\}
\\
\mathcal{K}_i^s&=\{K\subset \mathcal{N}^s\text{ compact and symmetric }\colon gen(K)\geq i\}
\\
\mathcal{K}_{i,k}^s&=\{K\subset \mathcal{N}_k^s\text{ compact and symmetric }\colon gen(K)\geq i \}
\\
c_i&=\sup_{K\in \mathcal{K}_i^s}\inf_{u\in K}\G(u)
\\
c_{i,k}&=\sup_{K\in \mathcal{K}_{i,k}^s}\inf_{u\in K}\G(u),
\end{align*}
observe that the critical levels $c_{i,k}$ are increasing in $k$ and choosing $S_i$ to be the unit sphere of $V_i$, we have that $gen(S_i)=i$ and $\inf_{u\in S_i}\mathcal{G}(u)>0$ therefore $c_{i,k}>0$ when $i\leq k$. If $i>k$ then $\mathcal{K}_{i,k}^s$ are empty and so $c_{i,k}=0$. The next lemma provides solution for the problem \eqref{eigenfinito}. 
\begin{lem}\label{LSDdimfin}
Let $i\in \N$ be given. Then there exist sequences $\{u_k\}_{k=i}^\infty\subset W_0^sE_M(\Omega)$, and $\{\lambda_k\}_{k=i}^\infty\subset (0,+\infty	)$ such that
\begin{align*}
u_k&\in \mathcal{N}_k^s\subset V_k
\\
\mathcal{M}_s^{'}(u_k)&=\lambda_k\mathcal{G}'(u_k)\qquad \text{in }V_k^{*}
\\
\mathcal{G}(u_k)&=c_{i,k},	
\end{align*}
for all $k=i,i+1,\dots$
\end{lem}
\begin{proof}
As we observe above for each $k$ we have $\M_s, \G \in  C^1(V_k)$ and $\mathcal{G}(0)=0$. 

On the other hand let $u\in V_k$ with $u\neq 0$ putting $r(u)=1/\|D^su\|_{M,\nu_n}$ we have

$$
\iint_{\R^{2n}}M(r(u) D^su )\,d\nu_n=1.
$$
Now from the finite dimensional Ljusternik-Schnirelmann theory (\cite[Theorem 2 and Corollary 7.1]{Zei1}) the lemma follows.
\end{proof}
We are interested in the study of the asymptotic behavior of the sequences $\{u_k\}_k$ and $\{\lambda_k\}_k$ in $k$ with fixed $i$. The following tools are needed first.
\begin{lem}\label{conv1}
Let $G_k\colon \R^{2n}\to \R$ be a sequence of functions such that $G_k\to G$ $\nu_n-$a.e. for some function $G\colon \R^{2n}\to \R$. Suppose that there exists a sequence of  functions $\{\Phi_k\}_k$ in $L^1(\R^{2n},\nu_n)$ such that 
\begin{equation}\label{cotaparaGk}
|G_k|\leq \Phi_k\quad   \nu_n-\text{a.e. in }\R^{2n}
\end{equation}

$\Phi_k\to \Phi$ $\nu_n-$a.e. in $\R^{2n}$ and 
$$
\iint_{\R^{2n}}\Phi_k\,d\nu_n\to \iint_{\R^{2n}}\Phi\,d\nu_n
$$
for some function $\Phi\in L^1(\R^{2n},\nu_n)$. Then 
$$
\iint_{\R^{2n}}G_k\,d\nu_n\to \iint_{\R^{2n}}G\,d\nu_n.
$$
\end{lem}
\begin{proof}
The condition \eqref{cotaparaGk} implies $ G_k+\Phi_k,\Phi_k-G_k\geq 0$ a.e. in $\R^{2n}$, and by Fatou's lemma we obtain 
\begin{align*}
\iint_{\R^{2n}}G\,d\nu_n+\iint_{\R^{2n}}\Phi\,d\nu_n&=\iint_{\R^{2n}}(G+\Phi)\,d\nu_n
\\
&\leq  \liminf_{k\to\infty}\iint_{\R^{2n}}(G_k+\Phi_k)\,d\nu_n
\\
&\leq\liminf_{k\to\infty}\iint_{\R^{2n}}G_k\,d\nu_n+\iint_{\R^{2n}}\Phi\,d\nu_n,
\end{align*} 
from where we obtain 
$$
\iint_{\R^{2n}}G\,d\nu_n\leq \liminf_{k\to\infty}\iint_{\R^{2n}}G_k\,d\nu_n. 
$$	
In similar way from $\Phi_k-G_k\geq 0$ a.e. in $\R^{2n}$ we obtain
\begin{align*}
\iint_{\R^{2n}}\Phi\,d\nu_n-\iint_{\R^{2n}}G\,d\nu_n &\leq \iint_{\R^{2n}}\Phi\,d\nu_n+\liminf_{k\to\infty}\left (-\iint_{\R^{2n}}G_k\,d\nu_n \right )
\\
&=\iint_{\R^{2n}}\Phi\,d\nu_n-\limsup_{k\to\infty}\iint_{\R^{2n}}G_k\,d\nu_n ,
\end{align*}
and so 
$$
\limsup_{k\to\infty}\iint_{\R^{2n}}G_k\,d\nu_n \leq \iint_{\R^{2n}}G\,d\nu_n.
$$
The proof is completed.
\end{proof}
\begin{lem}\label{conv2}
Let $G_k\colon \R^{2n}\to \R$ be a sequence of a nonnegative  functions in $L^1(\R^{2n},\nu_n)$ and $G\in L^1(\R^{2n},\nu_n)$ such that $G_k\to G$ $\nu_n-$a.e. in $\R^{2n}$ and 
$$\iint_{\R^{2n}} G_k\,d\nu_n\to \iint_{\R^{2n}}G\,d\nu_n.
$$	
Then $G_k\to G$ in $L^1(\R^{2n},\nu_n)$.
\end{lem}
\begin{proof}
If we define $\Phi_k:=G_k+G$ then $\Phi_k\in L^1(\R^{2n},\nu_n),\Phi_k\to 2G$ a.e. in $\R^{2n}$ and 
$$
\iint_{\R^{2n}}\Phi_k\,d\nu_n\to \iint_{\R^{2n}}2G\,d\nu_n.
$$
We observe that $|G_k-G|\leq \Phi_k$	a.e. in $\R^{2n}$, therefore by Lemma \ref{conv1}
$$
\iint_{\R^{2n}}|G_k-G|\,d\nu_n\to 0,
$$
and so $G_k\to G$ in $L^1(\R^{2n},\nu_n)$.
\end{proof}
We are in position to give the asymptotic behavior of the sequence $\{(\lambda_k,u_k)\}_k$ . 
\begin{thm}\label{teouilambdai}
Let $i\in \N$ be fixed and the sequences $\{u_k\}_{k=i}^\infty\in W_0^sE_M(\Omega)$ and $\{\lambda_k\}_{k=i}^\infty\in (0,+\infty)$ as given by Lemma \ref{LSDdimfin}. Then there exists $\bar{u}_i\in \Dom((-\Delta_m)^s)$ and $\bar{\lambda}_i\in (0,+\infty)$ such that up to a subsequence, $\lambda_k\to \bar{\lambda}_i$ and $u_k\to \bar{u}_i$ for $\sigma(W_0^sL_M(\Omega),W^{-s}E_{\bar{M}}(\Omega))$. Moreover, $\mathcal{M}_s(\bar{u}_i)=1,\mathcal{G}(\bar{u}_i)=\lim_{k\to\infty}c_{i,k}$, and $(\bar{\lambda}_i,\bar{u}_i)$ is an eigenpair for \eqref{intro.2}.
\end{thm}
\begin{proof}
Using that 
$$
\iint_{\R^{2n}}M(D^su_k)\,d\nu_n=1,
$$	
for all $k\geq i$ we have that the sequence $\{u_k\}_k$ is bounded in $W_0^sL_M(\Omega)$ therefore there exists $\bar{u}_i\in W_0^sL_M(\Omega)$ such that $u_k\to \bar{u}_i$ in $\sigma(W_0^sL_M(\Omega), W^{-s}E_{\bar{M}}(\Omega))$ for a subsequence then $\mathcal{G}(u_k)\to \mathcal{G}(\bar{u}_i)$ and so by Lemma \ref{LSDdimfin}  
$$
\mathcal{G}(\bar{u}_i)=\lim_{k\to\infty}\mathcal{G}(u_k)=\lim_{k\to\infty}c_{i,k},
$$
since $\{c_{i,k}\}_k$ are increasing and positive we have $\mathcal{G}(\bar{u}_i)>0$ implying $g(\bar{u}_i)\not\equiv 0$ (and therefore $\bar{u}_i\not\equiv0$). Then there exists $k_0>i$ and a function $\phi\in V_{k_0}$ such that
$$
\int_\Omega g(\bar{u}_i)(\bar{u}_i-\phi)\,dx<0.
$$
Now assume that $\lambda_k\to\infty$, when $k\to\infty$. Using the monotonicity of the operator $(-\Delta_m)^s$ for $k\geq k_0$
$$
\iint_{\R^{2n}}(m(D^su_k)-m(D^s\phi))(D^su_k-D^s\phi)\,d\nu_n\geq 0,
$$
the above inequality and Lemma \ref{LSDdimfin} produces
\begin{align*}
\iint_{\R^{2n}}m(D^s\phi)(D^su_k-D^s\phi)\,d\nu_n &\leq \iint_{\R^{2n}}m(D^su_k)(D^su_k-D^s\phi)\,d\nu_n
\\
&=\lambda_k\int_\Omega g(u_k)(u_k-\phi)\,dx,
\end{align*}
since 
$$
\int_\Omega g(u_k)(u_k-\phi)\,dx\to \int_\Omega g(\bar{u}_i)(\bar{u}_i-\phi)\,dx<0
$$
and 
$$
\iint_{\R^{2n}}m(D^s\phi)(D^su_k-D^s\phi)\,d\nu_n\to \iint_{\R^{2n}}m(D^s\phi)(D^s\bar{u}_i-D^s\phi)\,d\nu_n,
$$
we have a contradiction and so the sequence $\{\lambda_k\}$ is bounded. Therefore we may assume for a common subsequence that $\lambda_k\to \bar{\lambda}_i$ in $\R$ and $u_k\to \bar{u}_i$ for $\sigma(W_0^sL_M(\Omega), W^{-s}E_{\bar{M}}(\Omega))$.

Recall that Young's inequality implies that $\bar{M}(m(t))=m(t)t-M(t)\leq m(t)t$ forall $t\in \R$. Therefore
\begin{align*}
\iint_{\R^{2n}}\bar{M}(m(D^su_k))\,d\nu_n&\leq \iint_{\R^{2n}}	m(D^su_k)D^su_k\,\nu_n
\\
&=\lambda_k \int_\Omega g(u_k)u_k\,dx
\\
&\to \bar{\lambda}_i \int_\Omega g(\bar{u}_i)\bar{u}_i\,dx<\infty,
\end{align*}
where we used Lemma \ref{LSDdimfin}. Then the sequence $\{m(D^su_k)\}$ is bounded in $L_{\bar{M}}(\nu_n)$ and there exists $F\in L_{\bar{M}}(\nu_n)$ such that $m(D^s u_k)\to F$ in $\sigma(L_{\bar{M}}(\nu_n),E_{M}(\nu_n))$ therefore $(-\Delta_m)^su_k\to f=
\diver ^s F\in W^{-s}L_{\bar{M}}(\Omega)$ in $\sigma(W^{-s}L_{\bar{M}}(\Omega),W_0^sE_M(\Omega))$. For any $\phi\in V$ we have 
\begin{align*}
\langle f, \phi \rangle&=\lim_{k\to \infty}\langle (-\Delta_m)^su_k,\phi \rangle
\\
&=\lim_{k\to\infty}\iint_{\R^{2n}}m(D^su_k)D^s\phi\,d\nu_n
\\
&=\lim_{k\to\infty}\lambda_k\int_\Omega g(u_k)\phi\,dx
\\&=\bar{\lambda}_i\int_\Omega g(\bar{u}_i)\phi\,dx.
\end{align*}
Since $V$ is norm dense in  $W_0^sE_M(\Omega)$ and $W_0^s E_M(\Omega)$ is dense in $W_0^sL_M(\Omega)$ in $\sigma(L_M\times L_M(\nu_n),L_{\bar{M}}\times L_{\bar{M}}(\nu_n))$ it follows that 
\begin{equation}\label{eigen1}
\langle f, \phi \rangle=\bar{\lambda}_i\int_\Omega g(\bar{u}_i)\phi\,dx\qquad \forall \phi \in W_0^sL_M(\Omega).
\end{equation}
Also we have
\begin{align*}
\lim_{k\to\infty}\langle (-\Delta_m)^su_k,u_k \rangle&=\lim_{k\to\infty}\iint_{\R^{2n}}m(D^su_k)D^su_k\,d\nu_n
\\
&=\lim_{k\to\infty}\lambda_k\int_\Omega g(u_k)u_k\,dx
\\
&=\bar{\lambda}_i\int_\Omega g(\bar{u}_i)\bar{u}_i\,dx
\\
&=\langle f,\bar{u}_i \rangle.
\end{align*}
Now we are in position to apply the {\em pseudomonotonicity} property of  $(-\Delta_m)^s$ (Theorem \ref{pseudo}) and conclude that $\bar{u}_i\in \Dom((-\Delta_m)^s), (-\Delta_m)^s\bar{u}_i=f$ and 
\begin{equation}\label{paraL1}
\iint_{\R^{2n}}m(D^su_k)D^su_k\,d\nu_n\to \iint_{\R^{2n}}m(D^s\bar{u}_i)D^s\bar{u}_i\,d\nu_n.
\end{equation}
Therefore by \eqref{eigen1}
$$
\iint_{\R^{2n}}m(D^s\bar{u}_i)D^s \phi \,d\nu_n=\bar{\lambda}_i\int_\Omega g(\bar{u}_i)\phi\,dx\qquad \forall \phi\in W_0^s L_M(\Omega),
$$
that is $(\bar{\lambda}_i,\bar{u}_i)$ is an eigenpair for the problem \eqref{intro.2}. Observe that the above implies $\bar{\lambda}_i>0$.

On the other hand using Lemma \ref{conv2} and \eqref{paraL1} we have that $m(D^s u_k)D^s u_k\to m(D^s\bar{u}_i)D^s \bar{u}_i$ in $L^1(\R^{2n},\nu_n)$ and so there exists a majorant integrable $h\in L^1(\R^{2n},\nu_n)$ such that  $m(D^s u_k)|D^s u_k|\leq h$  $\nu_n-$a.e. in $\R^{2n}$ (\cite[Theorem 4.9]{brezis}), then by Young's inequality we have
\begin{align*}
M(D^s u_k) &\leq \bar{M}(m(D^su_k))+M(D^s u_k)
\\
&=m(D^s u_k)|D^s u_k|
\\
&\leq h,
\end{align*}
$\nu_n-$a.e. in $\R^{2n}$. By the compact immersion $W^s_0L_M(\Omega) \subset\subset E_M(\Omega)$ we can assume for a subsequence that $ u_k\to  \bar{u}_i$ a.e. in $\Omega$, and extending $\bar{u}_i=0$ in $\Omega^c$ we have $D^s u_k\to D^s \bar{u}_i$ $\nu_n-$a.e. in $\R^{2n}$, and consequently $M(D^s u_k)\to M(D^s \bar{u}_i)$ $\nu_n-$a.e. in $\R^{2n}$. Therefore by the dominated convergence theorem
$$
\mathcal{M}_s(\bar{u}_i)=\lim_{k\to\infty}\iint_{\R^{2n}}M(D^s u_k)\,d\nu_n=1.
$$
This finishes the proof of the theorem.
\end{proof}
In order to complete the proof of the main theorem we need to analize the asymptotic behavior of the eigenvalues $\bar{\lambda}_i$ as $i\to\infty$. In order to complet this fact we first study the behavior of the constants $c_{i,k}$ and $c_i$. This is the content of the next two lemmas. The proofs of this lemmas follow the same ideas of the lemmas 4.3 and 4.4 in \cite{Ti}. 
\begin{lem}\label{ciktoci}
Let $i\in\N$ be fixed. Then $c_{i,k}\to c_i$ as $k\to\infty$.
\end{lem}
\begin{proof}
Using the definition of the constants $c_{i,k}$ we have $c_{i,k}\leq c_{i,k+1}\leq \cdots \leq c_i$ for all $k\in \N$. We argue by contradiction, suppose that there exists $\varepsilon>0$ such that $c_{i,k}<c_i-\varepsilon$ for all  $k\in \N$. By the definition of $c_i$ and the supremum property there exists $K_\varepsilon\in \mathcal{K}_{i,k}^s$ such that
\begin{equation}\label{contra1}
c_i-\varepsilon/2< \inf_{w\in K_\varepsilon}\mathcal{G}(w).	
\end{equation}
We consider now the mappings $P_k\colon W_0^s E_M(\Omega)\to W_0^s E_M(\Omega) $ given by Theorem \ref{teo.abstracto}. It is easy to check that $0\notin P_k(K_\varepsilon)$ for every $k$ large enough. Indeed, suppose $P_{k_j}(w_{k_j})=0$ with $w_{k_j}\in  K_\varepsilon$ and $k_j\to \infty$. By compactness of the set $K_\varepsilon$, $w_{k_j}\to w\in K_\varepsilon$ for a subsequence. Theorem \ref{teo.abstracto} implies $w=0$ wich contradicts $0\notin K_\varepsilon$. Therefore the map $\Psi_k\colon K_\varepsilon \to \mathcal{N}^s\cap V_{m_k}$ defined as
$$
\Psi_k(w)=\frac{P_k(w)}{\|P_k(w)\|_{s,M}}
$$ 
is odd and continuous for $k$ large enough. Hence $\Psi_k(K_\varepsilon)\subset \mathcal{K}_{i,m_k}^s$ implying
$$
\inf_{w\in \Psi_k(K_\varepsilon)}\mathcal{G}(w)\leq c_{i,m_k},
$$
for every $k$ large enough. Thus, for every $k$ large enough there exists $w_k\in K_\varepsilon$ such that 
\begin{equation}\label{contra2}
\mathcal{G}(\Psi_k(w_k))<c_i-\varepsilon.
\end{equation}
by the compactness we have $w_k\to w\in K_\varepsilon$ in $W_0^sE_M (\Omega)$ for a subsequence implying, again by Theorem \ref{teo.abstracto}, that $\|P_k(w_k)\|_{s,M}\to \|w\|_{s,M}=1$ therefore 
$$
\Psi_k(w_k)\to w\qquad  \text{for }\sigma(W_0^s L_M(\Omega), W^{-s}E_{\bar{M}}(\Omega)), 
$$
and so $\mathcal{G}(\Psi_k(w_k))\to \mathcal{G}(w)$, which contradicts \eqref{contra1} and \eqref{contra2}.
\end{proof}
\begin{lem}\label{cito0}
$c_i\to 0$ as $i\to \infty$.
\end{lem}
\begin{proof}
Let $\varepsilon>0$ be arbitraty and $\{P_k\}_{k\in \N}$ the mappings given by Theorem \ref{teo.abstracto}. The continuity properties of the mappings $P_k$ and $\mathcal{G}$ imply the existence of $k_0\in \N$ and $\delta>0$ such that 
$$
|\mathcal{G}(P_{k_0}(w))-\mathcal{G}(w)|<\varepsilon/2\qquad \text{ for all }w\in  \mathcal{N}^s
$$
and
$$
\mathcal{G}(w)<\varepsilon/2\qquad \text{ for all }\| w\|_{M,s}\leq \delta.
$$ 
Therefore if $K\subset \mathcal{N}^s$ is compact and symmetric with $\inf_{w\in K}\mathcal{G}(w)>\varepsilon$ then
$$
\|P_{k_0}(w)\|_{M,s}\geq \delta\qquad \text{ for all }w\in K.
$$
Hence $gen(K)\leq gen (P_{k_0}(w))\leq m_{k_0}$ for some $m_{k_0}\in \N$. Thus, if $i>m_{k_0}$ and $K\in \mathcal{K}_i^s$ then 
$$
\inf_{w\in K}\mathcal{G}(w)\leq \varepsilon,
$$
and recalling that 
$$
c_i=\sup_{K\in  \mathcal{K}_i^s}\inf_{u\in K}\mathcal{G}(u)
$$
we have $c_i\leq \varepsilon$ for $i>m_{k_0}$.
\end{proof}
Now we are ready to complete the proof of the main result, namely Theorem \ref{teo.main}.
\begin{thm}
With the previous assumptions and notations,
$\bar{\lambda}_i\to +\infty$ and  $\bar{u}_i\to 0$ for $\sigma(W_0^sL_M(\Omega), W^{-s}E_{\bar{M}}(\Omega))$ as $i\to\infty$.
\end{thm}
\begin{proof}
Using Theorem \ref{teouilambdai}, Lemma \ref{ciktoci} and Lemma \ref{cito0} we obtain that $\mathcal{G}(\bar{u}_i)\to 0$ as $i\to\infty$. On the other by Theorem \ref{teouilambdai} we can assume that $\bar{u}_i\to \tilde{u}\in W_0^sL_M(\Omega)$ for $\sigma(W_0^sL_M(\Omega),W^{-s}E_{\bar{M}}(\Omega))$ as $i\to\infty$. From these facts we easily deduce that $\tilde{u}=0$.

Now taking $\bar{u}_i$ as test function in  \eqref{soldebil} we obtain
$$
\bar{\lambda}_i=\frac{\iint_{\R^{2n}}m(D^s \bar{u}_i)D^s \bar{u}_i\,d\nu_n}{\int_\Omega g(\bar{u}_i)\bar{u}_i\,dx}\geq \frac{1}{\int_\Omega g(\bar{u}_i)\bar{u}_i\,dx}.
$$
Finally by the compactness of the embedding $\bar{u}_i\to 0$ in $E_M$ as $i\to \infty$ from where it follows that $\bar{\lambda}_i\to \infty$ as $i\to\infty$.
\end{proof}
\appendix

\section{Density results}\label{ap.density}

In this Section, we prove some density results regarding the fractional order Orlicz-Sobolev spaces that are needed in this work. We mention that these results follow the same approach as the analog ones for the classical Orlicz-Sobolev spaces proved in \cite{Gossez}. So we only sketch the arguments, including some detail where the differences arise.

One of the key properties that we assume on the domain $\Omega$ is that it satisfies the so called {\em segment property}. 
\begin{defi}\label{segment}
Let $\Omega\subset\R^n$ be an open and bounded domain. We say that $\Omega$ has the {\em segment property} if there exists a locally finite open covering of $\partial \Omega$ with balls $\{B_t(x_j)\}$ centered in $x_j\in \partial \Omega$ with radius $t$, a corresponding sequence of units vectors $n_j$, and a number $t^*\in (0,1)$ such that
$$
x\in \bar{\Omega}\cap B_t(x_j)\Longrightarrow x+t n_j\in \Omega\text{ for all } t\in (0,t^*).
$$
\end{defi}
This condition about the domain $\Omega$ says, in some sense, that the domain lie locally on one side of its boundary. Observe that the segment property does not impply any smoothness on the boundary $\partial\Omega$. Conversely, if a domain is of class $C^1$ does not imply that the domain satisfy the segment property. See Section 8.1 in \cite{DB} for more details.

The main result in this section is the following.
\begin{thm}\label{teo.density}
Assume that $\Omega\subset\R^n$ has the segment property. Then ${\mathcal D}(\Omega)$ is dense in $W^s_0L_M(\Omega)$ with respect to the $\sigma(W^s_0L_M(\Omega), W^{-s}L_{\bar M}(\Omega))$ topology.
\end{thm}
In the rest of the paper we allways assume that the domain $\Omega$ satisfies the segment property.

For the proof of Theorem \ref{teo.density} we use a series of lemmas from \cite{Gossez}. 
\begin{lem}[\cite{Gossez}, Lemma 1.4]\label{lem1.4}
Let $u_k\in \LL_M(\R^N,d\mu)$ be such that $u_k\to u$ $\mu-$a.e. in $\R^N$ and $M(u_k)\le w_k$ $\mu-$a.e. in $\R^N$ where $w_k\to w$ in $L^1(\R^N,d\mu)$. Then $u\in \LL_M(\R^N,d\mu)$ and $u_k\to u$ in the $\sigma(L_M(\R^N,d\mu), L_{\bar M}(\R^N,d\mu))$ topology.
\end{lem}

\begin{lem}[\cite{Gossez}, Lemma 1.5]\label{lem1.5}
Let $u\in L_M(\R^N,d\mu)$ and denote by $u_t$ the translated function: $u_t(x) = u(x -t)$. Then, $u_t\to u$ in the $\sigma(L_M(\R^N,d\mu), L_{\bar M}(\R^N,d\mu))$ topology as $t\to 0$.
\end{lem}

\begin{lem}[\cite{Gossez}, Lemma 1.6]\label{lem1.6}
Let $u\in L_M(\R^n, dx)$ and denote by $u_\ve$ the regularized function: $u_\ve = u *\rho_\ve$, where $\rho_\ve(x) = \ve^{-n}\rho(x/\ve)$ is the standard mollifier. Then $u_\ve\to u$ in the $\sigma(L_M(\R^n,dx), L_{\bar M}(\R^n,dx))$ topology as $\ve\to 0$.
\end{lem}

Now, we need a modification of Lemma \ref{lem1.6} to deal with functions $F\in L_M(\R^{2n}, d\nu_n)$.

\begin{lem}\label{lem1.6bis}
Let $F\in L_M(\R^{2n}, d\nu_n)$ and denote by $F_\ve$ 
$$
F_\ve(x,y) = \int_{\R^n} F(x-z,y-z)\rho_\ve(z)\, dz,
$$ 
where $\rho_\ve(z) = \ve^{-n}\rho(z/\ve)$ is the standard mollifier. Then, if $F$ has compact support, $F_\ve\to F$ as $\ve\to 0$ in the $\sigma(L_M(\R^{2n},d\nu_n), L_{\bar M}(\R^{2n},d\nu_n))$ topology.
\end{lem}

\begin{proof}
The proof of this lemma is a modification of Lemma \ref{lem1.6}. In fact, first observe that it is enough to prove the lemma in the case where $F\in \LL_M(\R^{2n}, d\nu_n)$. Then, using Jensen's inequality, one can easily verify that
$$
M(F_\ve) \le M(F)_\ve = \int_{\R^n} M(F(x-z,y-z))\rho_\ve(z)\, dz.
$$
So, to finish the proof it remains to see that $M(F)_\ve\to M(F)$ in $L^1(\R^{2n}, d\nu_n)$, that $F_\ve\to F$ $\nu_n-$a.e. and apply Lemma \ref{lem1.4}.

First, observe that if $G\in L^1(\R^{2n},d\nu_n)$ then $G_\ve\to G$ in $L^1(\R^{2n},d\nu_n)$. In fact,
\begin{align*}
\iint_{\R^{2n}} |G_\ve - G|\, d\nu_n&\le \iint_{\R^{2n}} |G(x-z,y-z)-G(x,y)|\rho_\ve(z)\, dz\, d\nu_n(x,y)\\
&=\int_{|z|\le \ve} \rho_\ve(z)\left(\iint_{\R^{2n}}|G(x-z,y-z)-G(x,y)|\, d\nu_n(x,y)\right)\, dz.
\end{align*}
From this inequality, the result follows by the continuity of the $L^1-$norm.

Applying this result to $G=M(F)_\ve$ gives that $M(F)_\ve\to M(F)$ in $L^1(\R^{2n}, d\nu_n)$, Finally, observe that since $F$ has compact support, $F\in \LL_M(\R^{2n}, d\nu_n)$ implies that $F\in L^1(\R^{2n}, d\nu_n)$, then if we apply the same result to $G=F$ we get that $F_\ve\to F$ in $L^1(\R^{2n}, d\nu_n)$, so passing to a subsequence $\ve_k\to 0$ if necessary, we get the desired result.
\end{proof}

Now we are ready to prove the main result of this appendix
\begin{proof}[Proof of Theorem \ref{teo.density}]
Let $u\in W^s_0L_M(\Omega)$. We can assume, without loss of generality, that $u\in W^s_0L_M(\R^n)$ and that $u=0$ in $\R^n\setminus \Omega$.

Now, using Lemma \ref{lem1.4}, the segment property of $\Omega$, and observing that $D^su_t = (D^su)_t$, we can argue exactly as in the proof of \cite[Theorem 1.3]{Gossez} and assume, without loss of generality, that $u$ has compact support in $\Omega$.

Now, we can regularize $u$ by convolution $u_\ve=u*\rho_\ve$ and apply Lemmas \ref{lem1.6} and \ref{lem1.6bis} to conclude the desired result.
\end{proof}

\section{A Rellich-Kondrachov type result}
In this appendix we prove a Rellich-Kondrachov compactness result for the inclusion $W_0^s L_M(\Omega)\subset L_M(\Omega)$. In the case where $M$ satisfies the $\Delta_2-$condition this result was proof in \cite[Theorem 3.1]{FBS}.

It is worth mention that in \cite{ACPS2} optimal embeddings of the form $W^sL_M(\Omega)\subset L_N(\Omega)$ were obtained when $\Omega$ is Lipschitz and $M$ satisfies some subcritical conditions. See Theorem 9.1 in \cite{ACPS2}. 

The purpose of this appendix is to obtain the compact embedding result of \cite{FBS} whitout requiring the $\Delta_2-$condition on $M$. We want to stress that the main ideas in order to accomplish this task are taken from \cite{ACPS2}.
\begin{lem}
For all $u\in W^sL_M$ and $|h|<1/2$ we have
$$
\int_{\R^n}M\left (|u(x+h)-u(x)|\right)\,dx\leq \frac{2^{n+1}}{\omega_n}\iint_{\R^{2n}}M (2^{s+1}|h|^sD^s u)\,d\nu_n.
$$	
\end{lem}
\begin{proof}
Let $x,h\in \R^n$ with $|h|<1/2$ and $u\in W^sL_M$. We define the following sets 
\begin{align*}
S_1&=
 \left\{y\in B_{|h|}(x)\colon |u(x+h)-u(y)|\geq \frac{1}{2} |u(x+h)-u(x)|\right \},
 \\
S_2&= \left \{y\in B_{|h|}(x)\colon |u(x)-u(y)|\geq \frac{1}{2} |u(x+h)-u(x)|\right \}.
\end{align*}
Then $B_{|h|}(x)\subset S_1\cup S_2$. Therefore it follows that 
$$
|S_1|\geq \frac{1}{2} |B_{|h|}(x)|\quad \text{ or }\quad |S_2|\geq \frac{1}{2} |B_{|h|}(x)|.
$$
Without loss of generality we may assume that 
$$
\frac{1}{2}\omega_n |h|^n\leq |S_1|\leq \omega_n |h|^n.
$$
Hence we have
\begin{align*}
\int_{\R^n} M\left(|u(x+h)-u(x)| \right )\,dx&=\int_{\R^n}\frac{1}{|S_1|}\int_{S_1}M\left(|u(x+h)-u(x)| \right )\,dy\,dx
\\
&\leq \frac{2}{\omega_n|h|^n}\iint_{\R^n\times B_{|h|}(x)}M(2|u(x+h)-u(y)|)\,dy\,dx.
\end{align*}
The last integral is bounded by 
\begin{align*}
&\frac{2}{\omega_n|h|^n}\iint_{\R^n\times B_{|h|}(x)}M\left (2\frac{|u(x+h)-u(y)|}{|x+h-y|^s}|x+h-y|^s\right )|x+h-y|^n\frac{dydx}{|x+h-y|^n}
\\
&\leq  \frac{2^{n+1}}{\omega_n}\iint_{\R^{2n}}M\left ( 2^{s+1}|h|^s\frac{|u(x)-u(y)|}{|x-y|^s}\right )\,d\nu_n.
\end{align*}
This finish the proof.
\end{proof}
\begin{cor}\label{B2}
There exists a constant $C=C(n,s)>0$ such that
$$
\|\tau_h u-u\|_M\leq C|h|^s\|D^s u\|_{M,\nu_n},
$$ 	
for every $u\in W^sL_M$ and $|h|<1/2$
\end{cor}
\begin{proof}
Take $\lambda=\|D^s u\|_{M,\nu_n}2^{s+1}|h|^s A$ where $A=\max\{1,2^{n+1}/\omega_n\}$, and apply the previous lemma to the function $u/\lambda$. And we get
\begin{align*}
\int_{\R^n} M\left ( \frac{|u(x+h)-u(x)|}{\lambda}\right )\,dx&\leq \frac{2^{n+1}}{\omega_n}\iint_{\R^{2n}}M\left (\frac{D^s u}{\|D^s u\|_{M,\nu_n}A} \right )
\\
&\leq 1.
\end{align*}
This finish the proof taking $C=2^{s+1}A$.
\end{proof}
Whit these preliminaries we are ready to prove the main result of this appendix.	
\begin{thm}\label{teo.compact}
Let $\Omega\subset \R^n$ be a bounded	domain that satisfies the segment property and let $M$ be a Young function. Then the inclusion $W_0^sL_M(\Omega)\subset E_M(\Omega)$ is compact. That is if $\{u_k\}_{k\in\N}\subset W_0^s L_M(\Omega)$ is bounded, there exists $u\in E_M(\Omega)$ and a subsequence $\{u_{k_j}\}_{j\in \N}\subset \{u_k\}_{k\in\N}$ such that $u_{k_j}\to u$ in $E_M(\Omega)$ as $j\to\infty$. 
\end{thm}
\begin{proof}
With the help of Corollary \ref{B2} the proof of the theorem is an immediate consequence of \cite[Theorem 11.4]{KR}.
\end{proof}
 
\subsection*{Acknowledgements.} This work was partially supported by CONICET under grant  PIP 11220150100032CO and PIP 11220210100238CO, by ANPCyT under grants PICT 2019-3837 and PICT 2019-3530 and by University of San Luis under PROICO 03-2023.  Both authors are members of CONICET and are grateful for the support.

During the creation of this work, the second author benefitted from the collaborative environment at Instituto de Calculo -- CONICET, where a portion of the research was conducted. The gracious hospitality extended by the first author, along with the conducive work atmosphere, is sincerely acknowledged and greatly valued.

\bibliographystyle{amsplain}
\bibliography{biblio}

\end{document}